\documentclass[12pt]{extarticle}

\usepackage[margin=1.5in]{geometry}  
\usepackage{graphicx}              
\usepackage{amsmath}               
\usepackage{amsfonts}              
\usepackage{amsthm}                
\usepackage{framed}

\newtheorem{theorem}{Theorem}
\newtheorem{lemma}[theorem]{Lemma}
\newtheorem{remark}{Remark}

\begin{document}

\nocite{*}

\title{\bf On the Success of Mishandling \\ Euclid's Lemma}

\author{\textsc{Adrian W. Dudek} \\ 
Mathematical Sciences Institute \\
The Australian National University \\ 
\texttt{adrian.dudek@anu.edu.au}}
\date{}

\maketitle

\begin{abstract}
We examine Euclid's lemma that if $p$ is a prime number such that $p | ab$, then $p$ divides at least one of $a$ or $b$. Specifically, we consider the common misapplication of this lemma to numbers that are not prime, as is often made by undergraduate students. We show that a randomly chosen implication of the form $r |ab \Rightarrow r|a \text{ or } r|b$ is \textit{almost surely} false in a probabilistic sense, and we quantify this with a corresponding asymptotic formula.
\end{abstract}

I recently gave a tutorial to my undergraduate number theory class, wherein a severe warning was issued. Specifically, I told them that Euclid's lemma states that if $p$ is a prime and $p|ab$, then $p$ divides at least one of $a$ and $b$ (see Theorem 3 of Hardy and Wright \cite{hardywright}). I followed this with words of warning, grimly telling them that many students seem to let slip the all-important premise that $p$ must be prime. They nodded with spirit -- as undergraduate students often do -- but unfortunately and in spite of my stern words, a suite of assignments now sit on my desk, where Euclid's lemma has been heartily applied in the case where $p$ is not a prime. That is, many have incorrectly supposed that given any old divisor $r$ of the product $ab$, it must be the case that $r$ divides at least one of $a$ and $b$.

Of course, such an implication can sometimes be true, but is certainly not true in general (take, for example, $r=6$, $a=2$ and $b=3$). In fact, it is the purpose of this short article to show that a randomly chosen implication of the form
$$r | ab \Rightarrow r | a \text{ or } r | b$$
is \textit{almost surely} false, that is, the probability of it being true is zero. Moreover, we furnish an asymptotic formula which quantifies this exactly.

\begin{theorem} \label{main}
Let $B(N)$ count the number of triples $(a,b,r)$ of positive integers such that $r | ab$ and $ab \leq N$, and let $A(N)$ count the number of these triples such that $r$ divides at least one of $a$ and $b$. Then we have that
$$\frac{A(N)}{B(N)} \sim \frac{\pi^2}{\log N} $$
as $N \rightarrow \infty$.
\end{theorem}

First, let's calculate $B(N)$. As usual, we let $d(n)$ denote the number of divisors of~$n$. As $B(N)$ counts the number of triples $(a,b,r)$ such that $r | ab$ and $ab \leq N$, it is clear that
\begin{eqnarray*}
B(N) = \sum_{ab \leq N} d(ab)
\end{eqnarray*}
where the sum is over all $a$ and $b$ with $ab \leq N$. Therefore, we have that
\begin{eqnarray*}
B(N) = \sum_{ab \leq N} d(ab) & = & \sum_{n\leq N} \sum_{ab = n} d(ab) \\
& = & \sum_{n\leq N} d(n)^2 \\
& = & \pi^{-2} N \log^3 N + O(N \log^2 N)
\end{eqnarray*}
by a result of Ramanujan \cite{ramanujan}.

We now wish to bound $A(N)$. Writing this as
$$A(N) = \sum_{ab \leq N} \sum_{r | ab} 1_{ \{r | a \text{ or } r |b \} },$$
we first want to bound the number of divisors $r$ of $ab$ such that $r$ divides at least one of $a$ and $b$. Now, inclusion-exclusion gives us, for fixed integers $a$ and $b$, that
\begin{eqnarray*}
\sum_{r | ab} 1_{ \{r | a \text{ or } r |b \} } & = & d(a) + d(b) - \sum_{\substack{r|a \\ r |b}} 1.
\end{eqnarray*}
It follows immediately that
\begin{equation} \label{save}
A(N)  =  \sum_{ab \leq N} (d(a) + d(b))- \sum_{ab \leq N} \sum_{\substack{r|a \\ r |b}} 1.
\end{equation}
We will bound the sum on the right in the following lemma.
\begin{lemma} \label{bound}
We have that
$$\sum_{ab \leq N} \sum_{\substack{r|a \\ r |b}} 1 = O(N \log N).$$
\end{lemma}

\begin{proof}
Note that we are summing over all positive integers $rc$ and $re$ such that $r^2 ce \leq N$. Thus, we have that
$$\sum_{ab \leq N} \sum_{\substack{r|a \\ r |b}} 1 = \sum_{r \leq N} \sum_{rc \leq N} \sum_{re \leq N/rc} 1,$$
where the second and third summations are over $c$ and $e$ respectively. Using the fact that $\sum_{n \leq x} n^{-1} = \log x +O(1)$, we have that
\begin{eqnarray*}
\sum_{r \leq N} \sum_{rc \leq N} \sum_{re \leq N/rc} 1 & = & \sum_{r \leq N} \sum_{c \leq N/r} \bigg(\frac{N}{r^2c} + O(1) \bigg) \\
& = & N \sum_{r \leq N} \frac{\log(N/r)}{r^2} +O(N \log N) \\
& = & N \log N \sum_{r \leq N} \frac{1}{r^2} - N \sum_{r \leq N} \frac{\log r}{r^2} + O(N \log N)
\end{eqnarray*}
Noting that both of the sums in the above equation are $O(1)$ completes the proof.
\end{proof}

\begin{remark}
From the fact that $r|a$ and $r|b$ if and only if $r | gcd(a,b)$, it follows that the sum in Lemma \ref{bound} is equal to
$$\sum_{ab \leq N} d(\gcd(a,b)).$$
It would be interesting to see an asymptotic formula for this.
\end{remark}

Now, from the symmetry of divisors, (\ref{save}) and the above lemma, it follows that
\begin{eqnarray*}
A(N) & = & 2 \sum_{ab \leq N} d(a) + O(N \log N) \\
& = & 2 \sum_{b \leq N} \sum_{a \leq N/b} d(a) +O(N \log N).
\end{eqnarray*}
Using the classic bound (see Theorem 320 of \cite{hardywright})
$$\sum_{n \leq N} d(n) = N \log N + O(N),$$
we have that
\begin{eqnarray*}
A(N) & =  & 2 \sum_{b \leq N} \bigg( \frac{N}{b} \log \frac{N}{b} + O\bigg(\frac{N}{b}\bigg) \bigg)+O(N \log N)\\
& = & 2 N \log N \sum_{b \leq N} \frac{1}{b} - 2 N \sum_{b \leq N} \frac{\log b}{b} +O(N \log N).
\end{eqnarray*}
By comparison to the integral, we can estimate the middle sum \textit{viz.}
\begin{eqnarray*}
\sum_{b \leq N} \frac{\log b}{b} & = & \int_1^N \frac{\log t}{t} dt + O\bigg( \frac{\log N}{N} \bigg) \\
& = & \frac{\log^2 N}{2} + O\bigg( \frac{\log N}{N} \bigg).
\end{eqnarray*}
It follows that
$$A(N) = N \log^2 N + O(N \log N)$$
and this completes the proof of Theorem \ref{main}.

\subsection*{Acknowledgements}
The author is gracious of the financial support provided by an Australian Postgraduate Award and an ANU Supplementary Scholarship. He would also like to thank the anonymous referees for providing useful feedback, and his number theory class for accidentally motivating the result of this paper.

\bibliographystyle{plain}

\bibliography{biblio}

\end{document}